\documentclass[11pt]{article}
\usepackage{mathrsfs}
\usepackage{bbm}
\usepackage{amsfonts}
\usepackage{amsmath,amssymb}
\usepackage{amsmath}
\usepackage{amssymb}
\usepackage{graphicx}

\newtheorem{theorem}{Theorem}[section]
\newtheorem{corollary}[theorem]{Corollary}
\newtheorem{definition}[theorem]{Definition}
\newtheorem{example}[theorem]{Example}
\newtheorem{lemma}[theorem]{Lemma}
\newtheorem{proposition}[theorem]{Proposition}
\newtheorem{remark}[theorem]{Remark}

\newenvironment{proof}[1][Proof]{\noindent \textbf{#1.} }{\  $\Box$}
\numberwithin{equation}{section}

\begin{document}

\title{\textbf{$G$-Expectation Weighted Sobolev Spaces, Backward SDE and
Path Dependent PDE}}
\author{Shige Peng\thanks{%
School of Mathematics and Qilu Institute of Finance, Shandong University,
peng@sdu.edu.cn, Peng's research was partially supported by NSF of China No.
10921101; and by the 111 Project No. B12023} \and Yongsheng Song\thanks{%
Academy of Mathematics and Systems Science, CAS, Beijing, China,
yssong@amss.ac.cn. Research supported by by NCMIS; Youth Grant of National
Science Foundation (No. 11101406); Key Lab of Random Complex Structures and
Data Science, CAS (No. 2008DP173182).} }
\maketitle
\date{}

\begin{abstract}
We introduce a new notion of $G$-expectation-weighted Sobolev spaces, or in
short, $G$-Sobolev spaces, and provide a 1-1 correspondence between a type
of backward SDEs driven by $G$-Brownian motion
\begin{equation*}
Y_{t}=\xi+\int_{t}^{T}f(s,Y_{s},Z_{s},\eta_{s})ds-%
\int_{t}^{T}Z_{s}dB_{s}-(K_{T}-K_{t}),
\end{equation*}
where $K_{t}=\frac{1}{2}\int_{0}^{t}\eta_{s}d\langle
B\rangle_{s}-\int_{0}^{t}G(\eta_{s})ds$, and a type of path dependent PDEs
in the corresponding  $G$-Sobolev space $W_{G}^{1,2;p}(0,T)$
\begin{align*}
D_{t}u+G(D_{x}^{2}u)+f(t,u,D_{x}u,D_{x}^{2}u) & =0,\  \ t\in \lbrack0,T), \\
u(T, \omega)) & =\xi(\omega).
\end{align*}
When  $f$ is independent of $D_{x}^{2}u$, we also formulate a
different type of weak solution to the PPDEs in the  $G$-Sobolev
space $W_{\mathcal{A}_G}^{\frac{1}{2},1;p}(0,T)$, which is an expansion of  $W_{G}^{1,2;p}(0,T)$ with weaker derivatives. This weak formulation corresponds exactly to the $G$-BSDEs studied in [HJPS12]. Particularly, for the linear case of $G$ corresponding to the classical Wiener probability
space $(\Omega,\mathcal{F},P)$, we have established a 1-1 correspondence
between BSDEs and a type of quasilinear path dependent PDEs, whose solutions
are defined in a Sobolev space weighted by the Wiener measure. 
\end{abstract}


\textbf{MSC-classification}: 60H, 60E, 62C, 62D, 35J, 35K.


\section{Introduction}

The notion of path dependent PDEs was proposed in Peng's ICM2010 lecture. In
fact for fully nonlinear path dependent PDEs corresponding to $G$%
-martingales, the motivation was already revealed in the construction of $G$%
-expectation. The main idea of the construction of a $G$-expectation can be
seen as to solve the following well-posed classical PDE (called $G$-heat
equation), for a given sublinear function $G(a)=\frac{1}{2}(\overline{\sigma}%
^{2}a^{+}-\underline {\sigma}^{2}a^{-})$ defined on $a\in \mathbb{R}$,%
\begin{equation}  \label{NG}
\partial_{t}u(t,x)+G(\partial^2_{x}u(t,x))=0,\  \  \ t\in \lbrack0,T),\  \ x\in
\mathbb{R}
\end{equation}
with a given terminal condition $u(T,x)=\varphi(x)$, where $\varphi$ is a
Lipschitz function. For readers not familiar with this type of
fully nonlinear PDEs, we suggest to consider a linear function $G(a)=\frac{a%
}{2}$, corresponding to the case $\overline{\sigma}^{2}=\underline{\sigma}%
^{2}=1$. In this case the $G$-heat equation becomes%
\begin{equation}  \label{LG}
\partial_{t}u(t,x)+\frac{1}{2}\partial^2_{x}u(t,x)=0,\  \  \ t\in \lbrack
0,T),\  \ x\in \mathbb{R}
\end{equation}
and the corresponding solution is the following smooth function:%
\begin{equation}  \label{LG-formula}
u(t,x)=\frac{1}{\sqrt{2\pi}}\int_{-\infty}^{\infty}\varphi(x+\sqrt{t}y)\exp(-%
\frac{y^{2}}{2})dy.
\end{equation}

An important point of view is to understand the solution $u$ as a stochastic
process, or a function of paths $\bar{u}(t,\omega)=u(t,\omega(t))_{t\in
\lbrack0,T]}$ defined on $\Omega_{T}=C_{0}([0,T],\mathbb{R})$, the
collection of $\mathbb{R}$-valued continuous functions $\omega$ on $[0,T]$
with $\omega(0)=0$. The corresponding PDE is a special case of the following
path dependent PDE
\begin{equation}  \label{NG-path}
D_{t}\bar{u}(t,\omega)+G(D^2_{x}\bar{u}(t,\omega))=0,\  \  \ t\in \lbrack
0,T),\  \  \omega \in{\normalsize \Omega}_{T}
\end{equation}
with terminal condition $\bar{u}(T,\omega)=\varphi(\omega(T))$. Here the
derivatives of of $u(\cdot,\omega)$ is trivially defined as
\begin{align}  \label{opera1}
D_{t}\bar{u}(t,\omega) & =\partial_{t}u(t,x)|_{x=\omega(t)},\  \  \\
D_{x}\bar {u}(t,\omega)& =\partial_{x}u(t,x)|_{x=\omega(t)}, \  \ D^2_{x}\bar{%
u}(t,\omega) =\partial^2_{x}u(t,x)|_{x=\omega(t)}.\  \  \   \label{opera2}
\end{align}
For a more general function of paths
\begin{equation*}
\bar{u}(t,\omega)=u_k(t,x;\omega(t_{1}),\cdots,\omega(t_{k}))|_{x=%
\omega(t)},\ t\in[t_{k},t_{k+1}]
\end{equation*}
with $t_{1}<\cdots<t_{k}$, $D_{t}\bar{u}(t,\omega)$ and $D_{x}^{2}\bar{u}%
(t,\omega)$ are defined similarly for the function $u_k$ of $(t,x)$
parameterized by $\omega(t_{1})$,$\cdots ,\omega(t_{k})$. Here $u_k$, $%
k=0,\cdots,n-1$, are smooth functions on $[t_k, t_{k+1}]\times R^{k+1}$%
satisfying
\begin{equation*}
u_{k}(t_{k+1},x;x_{1},\cdots,x_{k})=u_{k+1}(t_{k+1},x;x_{1},\cdots x_{k},x).
\end{equation*}
The definition of derivatives for the above smooth cylinder path process $%
u(t,\omega)$ corresponds perfectly with Dupire's one.

Now this path dependent PDE (\ref{NG-path}) is well-defined on the path
space $\Omega_{T}$ with terminal condition $\bar{u}(T,\omega)=\xi \in
L_{ip}(\Omega_{T})$, the space of all real-valued functions of paths of the
form $\xi(\omega )=\varphi(\omega(t_{1}),\cdots,\omega(t_{n}))$, $%
0<t_{1}<\cdots<t_{n}=T$, $n=1,2,\cdots$. Here $\varphi$ is a locally Lipschitz
function on $\mathbb{R}^{n}$. We can solve the PPDE (\ref{NG-path}) with a
given terminal condition $\xi \in L_{ip}(\Omega_{T})$. In fact for each $\xi
\in L_{ip}(\Omega_{T})$, there exists a unique (viscosity) solution $\bar{u}%
^{\xi}(t,\omega)_{t\in \lbrack0,T]}$ with terminal condition $\bar{u}%
^{\xi}(T,\omega)=\xi(\omega)$ such that $\bar{u}(t,\omega)\in
L_{ip}(\Omega_{t})$.

An important step in the $G$-expectation theory is to define $\bar{u}^{\xi
}(0,\omega)$ as a functional $\mathbb{E}^{G}$:
\begin{equation*}
\mathbb{E}^{G}{\normalsize [\xi]:=\bar{u}}^{\xi}{\normalsize (0,\omega):L}%
_{ip}(\Omega_{T})\mapsto \mathbb{R}.
\end{equation*}
In fact $\mathbb{E}^{G}$ forms a sublinear expectation defined on the space
of random variables ${\normalsize L}_{ip}(\Omega_{T})$. Moreover, the
corresponding time-conditional expectation is also well defined by%
\begin{equation*}
\mathbb{E}_{t}^{G}{\normalsize [\xi]:=\bar{u}(t,\omega):L}_{ip}(\Omega
_{T})\mapsto{\normalsize L}_{ip}(\Omega_{t}),\  \ t\in \lbrack0,T].
\end{equation*}
Furthermore, since this expectation is sublinear, we can naturally
introduce, for each $p\geq1$, a norm $\left \Vert \xi \right \Vert
_{L_{G}^{p}}:=\mathbb{E}^{G}{\normalsize [|\cdot|}^{p}{\normalsize ]}^{1/p}$
on ${\normalsize L}_{ip}(\Omega_{T})$ by which the completion $%
L_{G}^{p}(\Omega_{T})$ is a Banach space and thus a sublinear expectation
space $(\Omega_{T},L_{G}^{p}(\Omega_{T}),\mathbb{\hat{E}})$ is well defined.
Now for a given $\xi \in L_{G}^{p}(\Omega_{T})$, we can define,
\begin{equation*}
u(t,\omega)=\mathbb{E}_{t}^{G}[\xi](\omega),
\end{equation*}
which can be defined as a generalized solution of the path dependent
equation (\ref{NG-path}) with terminal condition $u(T,\omega)=\xi(\omega)\in
L_{G}^{p}(\Omega_{T})$.

A very interesting property is that, just like a Brownian motion, the
canonical process $B_{t}(\omega):=\omega(t)$ is still a continuous process
with stationary and independent increments. We call it a $G$-Brownian motion
associated to the function $G$. A new type of stochastic calculus for a $G$%
-It\^{o} process
\begin{equation}  \label{G-Ito}
u(t,\omega)=u(0,\omega)+\int_{0}^{t}\eta(s,\omega)ds+\int_{0}^{t}\zeta
(s,\omega)dB_{s}+\int_{0}^{t}\gamma(s,\omega)d\left \langle B\right \rangle
_{s}
\end{equation}
has been established, where $\left \langle B\right \rangle $ is the
quadratic variation process of $B$, which is still a continuous process with
stationary and independent increments. Here $\eta$ and $\gamma$ (resp. $\zeta$) are
\textquotedblleft adapted\textquotedblright \ processes in the completed
space $M_{G}^{p}(0,T)$ (resp. $H_G^p(0,T)$) under the norm $\left \Vert \cdot \right \Vert
_{M_{G}^{p}}$ (resp. $\left \Vert \cdot \right \Vert
_{H_{G}^{p}}$) mimic to that under classical norms $\left \Vert \cdot \right \Vert
_{M_{P}^{p}}$ (resp. $\left \Vert \cdot \right \Vert
_{H_{P}^{p}}$) but with the classical Wiener expectation $E_{P}[\cdot]$  replaced by
the sublinear expectation $\mathbb{E}^{G}[\cdot]$. Observe that we can also
use the norm $\left \Vert \cdot \right \Vert _{M_{G}^{p}}$ (resp. $\left \Vert \cdot \right \Vert
_{H_{G}^{p}}$) to the derivative processes $%
D_{t}u(t,\omega)$ and $D_{x}^{2}u(t,\omega)$ (resp. $D_{x}u(t,\omega)$) defined in (%
\ref{opera1}) and (\ref{opera2}) and then take the extension under those new
Sobolev norms. This procedure provides us  very interesting $\mathbb{E}^{G}$%
-weighted Sobolev spaces. The corresponding fully nonlinear PPDEs (\ref%
{NG-path}) can be well defined in this new framework.

Surprisingly, this framework is closely related to the $G$-It\^{o} process (%
\ref{G-Ito}). Let us first consider a simple case $u(t,\omega )=\varphi
(t,\omega (t))=\varphi (t,B_{t}(\omega ))$ for a smooth function $\varphi
\in C_{b}^{\infty }(\mathbb{R}_{+}\times \mathbb{R})$, for which the $G$-It%
\^{o}'s formula is
\begin{equation*}
\varphi (t,B_{t})=\varphi (0,0)+\int_{0}^{t}\partial _{s}\varphi
(s,B_{s})ds+\int_{0}^{t}\partial _{x}\varphi (s,B_{s})dB_{s}+\int_{0}^{t}%
\frac{1}{2}\partial _{x}^{2}\varphi (s,B_{s})d\left \langle B\right \rangle
_{s},
\end{equation*}%
or, if we use the derivatives of path functions defined in (\ref{opera1})
and (\ref{opera2})
\begin{equation}
u(t)=u(0)+\int_{0}^{t}D_{s}u(s)ds+\int_{0}^{t}D_{x}u(s)dB_{s}+\int_{0}^{t}%
\frac{1}{2}D_{x}^{2}u(s)d\left \langle B\right \rangle _{s}.
\label{G-Ito-formula}
\end{equation}%
In this paper we will see that, with the about defined weak derivatives, any well defined $G$-It\^{o} process $%
u(t,\omega )$ of the form (\ref{G-Ito}) corresponds uniquely to
\begin{equation*}
\eta (s,\omega )=D_{s}u(s,\omega ),\  \  \zeta (s,\omega )=D_{x}u(s,\omega ),\
\  \gamma (s,\omega )=\frac{1}{2}D_{x}^{2}u(s,\omega ).
\end{equation*}%
This implies that, with this $G$-Sobolev formulation, any $G$-It\^{o}
process $u(t,\omega )$ is in fact a generalized $G$-It\^{o}'s formula. The
above result also confirmed our belief that it is important to clearly
distinguish the $dt$ part and $d\left \langle B\right \rangle _{t}$ part for
a $G$-It\^{o} process (\ref{G-Ito}). This point is now well understood
thanks to Theorem 3.3 of Song (2012), which plays an important
role for the above 1-1 correspondence.

In the linear case $G(a)=\frac{a}{2}$, the space $(\Omega
_{T},L_{G}^{p}(\Omega _{T}),\mathbb{E}^{G})$ coincides with the classical
Wiener probability space $(\Omega ,\mathcal{F},P)$, in the sense that $%
\mathbb{E}^{G}=E_{P}$ and $B$ becomes a standard Brownian motion under the
Wiener probability measure $P$. Under this weaker expectation one cannot
distinguish the process $\left \langle B\right \rangle _{t}$ from $t$, and
thus corresponding to (\ref{G-Ito}), the It\^{o} process becomes%
\begin{equation*}
u(t,\omega )=u(0,\omega )+\int_{0}^{t}\beta (s,\omega )ds+\int_{0}^{t}\zeta
(s,\omega )dB_{s},
\end{equation*}%
where $\beta $ comes from the sum \textquotedblleft $\eta +\gamma $%
\textquotedblright \ (see (\ref{G-Ito})) and we thus have
\begin{equation*}
\beta _{s}=D_{s}u(s,\omega )+\frac{1}{2}D_{x}^{2}u(s,\omega ),\  \zeta
(s,\omega )=D_{x}u(s,\omega ).\
\end{equation*}%
We then see that even within this classical framework of Wiener
probability space, each It\^{o} process provides the corresponding
path dependent It\^{o}'s formula in the corresponding Sobolev space.
We can also formulate the path dependent PDE of the heat equation
\begin{equation}
D_{t}u(t,\omega )+\frac{1}{2}D_{x}^{2}u(t,\omega )=0,\  \  \ (t,\omega )\in
\lbrack 0,T)\times \Omega _{T}.  \label{LG-path}
\end{equation}

In this paper we will formulate several important spaces the under $G$%
-expectation to obtain the well-posedness of certain fully nonlinear PPDEs.
After many years of explorations, some elegant results and powerful tools
were obtained within the $G$-expectation framework. One of the main
objectives of this paper is to establish some important 1-1 correspondence between BSDEs and PPDEs
by which those results can be directly applied to the corresponding problems
of path dependent PDEs.

Observing that when we study the path-derivatives $D_{t}$ and $D_{x}^{(k)}$
in the Sobolev spaces, many \textquotedblleft mysterious\textquotedblright \
phenomena happen. For example for the path process $v(t,\omega
):=\left
\langle B\right \rangle _{t}$ we have%
\begin{equation*}
D_{t}v(t,\omega )\equiv 0,\ D_{x}v(t,\omega )\equiv 0\text{, but }\
D_{x}^{2}v(t,\omega )\equiv 2.
\end{equation*}%
We provide a 1-1 correspondence between a type of backward SDEs driven by $G$%
-Brownian motion
\begin{equation*}
Y_{t}=\xi +\int_{t}^{T}f(s,Y_{s},Z_{s},\eta
_{s})ds-\int_{t}^{T}Z_{s}dB_{s}-(K_{T}-K_{t}),
\end{equation*}%
where $K_{t}=\frac{1}{2}\int_{0}^{t}\eta _{s}d\langle B\rangle
_{s}-\int_{0}^{t}G(\eta _{s})ds$, and the following type of fully nonlinear
path dependent PDEs in the corresponding  $G$-Sobolev space $W_{G}^{1,2;p}(0,T)$
\begin{align}
D_{t}u+G(D_{x}^{2}u)+f(t,u,D_{x}u,D_{x}^{2}u)& =0,\  \ t\in \lbrack 0,T),
\label{int1} \\
u(T,\omega ))& =\xi (\omega ).  \label{int2}
\end{align}

When  $f$ is independent of $D_{x}^{2}u$, we also
formulate a different type of weak solution to the PPDEs (\ref{int1}-\ref%
{int2}) in the  $G$-Sobolev
space $W_{\mathcal{A}_G}^{\frac{1}{2},1;p}(0,T)$, which is an expansion of  $W_{G}^{1,2;p}(0,T)$ with weaker derivatives
\begin{align}
\mathcal{A}_{G}u+f(t,u(t,\omega ),D_{x}u(t,\omega ))& =0,\  \ t\in \lbrack
0,T),  \label{int3} \\
u(T,\omega ))& =\xi (\omega ).  \label{int4}
\end{align}%
This weak formulation corresponds exactly to $G$-BSDEs studied in [HJPS12]. Consequently,
the existence and uniqueness of weak solutions to the path dependent PDEs (\ref {int3}-\ref {int4}) have been obtained via the result of $G$-BSDEs.

The notion of path dependent PDEs was proposed by [Peng2010-2012] and the
nonlinear Feynman-Kac formula for a system of quasilinear path dependent
PDEs via BSDE approach was obtained in [PengWang2011]. A notion of viscosity
solutions of Dupire's type was proposed in [Peng2012]. A new notion of
viscosity solutions was studied by [Ekhen et al2011].

The paper is organized as follows. In section 2, we present some basic
notions and definitions of the related spaces under $G$-expectation. In
section 3 we define a weaker Sobolev space $W_{\mathcal{A}}^{\frac{1}{2},1;p}(0,T)$ in
the Wiener probability space. The Sobolev space $W_{G}^{1,2;p}(0,T)$
weighted by the sublinear $G$-expectation space and the 1-1 correspondence
between BSDEs and PPDEs are studied in section 4. In section 5 we define a
Sobolev space $W_{\mathcal{A}_G}^{\frac{1}{2},1;p}(0,T)$, which is an expansion of  $W_{G}^{1,2;p}(0,T)$ with weaker derivatives. In the space $W_{\mathcal{A}_G}^{\frac{1}{2},1;p}(0,T)$ we formulate the weak solutions to the
PPDEs. The recent results of existence and uniqueness of $G$-BSDEs (see
Appendix) are directly applied to the corresponding path dependent PDEs.

\section{Some definitions and notations}

We review some basic notions and definitions of the related spaces under $G$%
-expectation. The readers may refer to \cite{P07a}, \cite{P07b}, \cite{P08a}%
, \cite{P08b}, \cite{P10} for more details.

Let $\Omega=C_{0}(\mathbb{R}^{+};\mathbb{R}^{d})$ be the space of all $%
\mathbb{R}^{d}$-valued continuous paths $\omega=(\omega(t))_{t\geq0}\in
\Omega$ with $\omega(0)=0$ and let $B_{t}(\omega)=\omega(t)$ be the
canonical process. For $t\in \mathbb{R}^{+}$, we denote
\begin{equation*}
\Omega_{t}=\{(\omega(s\wedge t))_{s\geq0}:\omega \in \Omega \}.
\end{equation*}
Let us recall the definitions of $G$-Brownian motion and its corresponding $%
G $-expectation introduced in [Peng2007]. We are given a linear space of
functions of paths:
\begin{equation*}
L_{ip}(\Omega_{T}):=\{
\varphi(\omega(t_{1}),\cdots,\omega(t_{n})):t_{1},\cdots,t_{n}\in
\lbrack0,T],\  \varphi \in C_{l.Lip}((\mathbb{R}^{d})^{n}),\ n\in \mathbb{N}%
\},
\end{equation*}
where $C_{l.Lip}(\mathbb{R}^{n})$ is the collection of locally Lipschitz
functions on $\mathbb{R}^{n}$.

We are given a function%
\begin{equation*}
G:\mathbb{S}(d)\mapsto \mathbb{R}
\end{equation*}
satisfying the following monotonicity and sublinearity:

\begin{description}
\item[a.] \bigskip$G(a)\geq G(b),\  \ $if $a,b\in \mathbb{S}{\normalsize (d)}$
and $a\geq b;$

\item[b.] $G(a+b)\leq G(a)+G(b),\  \ {\normalsize \ G(\lambda a)=\lambda
G(a),\  \ }$for each $a,b\in \mathbb{S}{\normalsize (d)}$ and $\lambda \geq0.$
\end{description}

\begin{remark}
When $d=1$, we have $G(a):=\frac{1}{2}(\overline{\sigma}^{2}a^{+}-\underline{%
\sigma}^{2}a^{-})$, for $0\leq \underline{\sigma}^{2}\leq \overline{\sigma}%
^{2}$. We are also interested in the linear function $G(a)=a/2$.
\end{remark}

\bigskip \ For each $\xi(\omega)\in L_{ip}(\Omega_{T})$ of the form
\begin{equation*}
\xi(\omega)=\varphi(\omega(t_{1}),\omega(t_{2}),\cdots,\omega(t_{n})),\  \
0=t_{0}<t_{1}<\cdots<t_{n}=T,
\end{equation*}
we define the following $G$-conditional expectation
\begin{equation*}
\mathbb{E}_{t}^{G}[\xi]:=u_{k}(t,\omega(t);\omega(t_{1}),\cdots,\omega
(t_{k-1}))
\end{equation*}
for each $t\in \lbrack t_{k-1},t_{k})$, $k=1,\cdots,n$. Here, for each $%
k=1,\cdots,n$, $u_{k}=u_{k}(t,x;x_{1},\cdots,x_{k-1})$ is a function of $%
(t,x)$ parameterized by $(x_{1},\cdots,x_{k-1})\in \mathbb{R}^{k-1}$, which
is the solution of the following PDE ($G$-heat equation) defined on $%
[t_{k-1},t_{k})\times \mathbb{R}$:
\begin{equation*}
\partial_{t}u_{k}+G(\partial^2_{x}u_{k})=0\
\end{equation*}
with terminal conditions
\begin{equation*}
u_{k}(t_{k},x;x_{1},\cdots,x_{k-1})=u_{k+1}(t_{k},x;x_{1},\cdots x_{k-1},x),
\, \, \hbox{for $k<n$}
\end{equation*}
and $u_{n}(t_{n},x;x_{1},\cdots,x_{n-1})=\varphi (x_{1},\cdots x_{n-1},x)$.

The $G$-expectation of $\xi(\omega)$ is defined by $\mathbb{E}^{G}[\xi]=%
\mathbb{E}_{0}^{G}[\xi]$. From this construction we obtain a natural norm $%
\left \Vert \xi \right \Vert _{L_{G}^{p}}:=\mathbb{E}^{G}[|\xi|^{p}]^{1/p}$.
The completion of $L_{ip}(\Omega_{T})$ under $\left \Vert \cdot \right \Vert
_{L_{G}^{p}}$ is a Banach space, denoted by $L_{G}^{p}(\Omega_{T})$. The
canonical process $B_{t}(\omega):=\omega(t)$, $t\geq0$, is called a $G$%
-Brownian motion in this sublinear expectation space $(\Omega,L_{G}^{p}(%
\Omega ),\mathbb{E}^G)$.

The process $u(t,\omega ):=\mathbb{E}_{t}^{G}[\xi ](\omega )$, $t\in \lbrack
0,T]$, is a $G$-martingale, which is regarded as a typical solution of the
path dependent equation of \textquotedblleft $D_{t}u+G(D_{x}^{2}u)=0$%
\textquotedblright \ with terminal condition $u(T,\omega )=\xi (\omega )$.
In fact, in this formulation, the construction of a \textquotedblleft $G$%
-Sobolev space\textquotedblright \ is already implicitly given.

\begin{definition}
\label{def2.1} (Cylinder function of paths)

A function $\xi:\Omega_{T}\rightarrow \mathbb{R}$ is called a cylinder
function of paths on $[0,T]$ if it can be represented by
\begin{equation*}
\xi(\omega)=\varphi(\omega(t_{1}),\cdot \cdot \cdot,\omega(t_{n})),\,\,\,\omega
\in \Omega_{T},
\end{equation*}
for some $0<t_{1}<\cdot \cdot \cdot<t_{n}\leq T$, where $\varphi:(\mathbb{R}%
^{d})^{n}\rightarrow \mathbb{R}$ is a $C^{\infty}$-function with at most
polynomial growth. We denote by $C^{\infty}(\Omega_{T})$ the collection of
all cylinder functions of paths on $[0,T]$.
\end{definition}

\begin{definition}
\label{def2.2} (Step process)

A function $\eta (t,\omega ):[0,T]\times \Omega _{T}\rightarrow \mathbb{R}$
is called a step process if there exists a time partition $%
\{t_{i}\}_{i=0}^{n}$ with $0=t_{0}<t_{1}<\cdot \cdot \cdot <t_{n}=T$, such
that for each $k=0,1,\cdot \cdot \cdot ,n-1$ and $t\in (t_{k},t_{k+1}]$
\begin{equation*}
\eta (t,\omega )=\varphi _{k}(\omega (t_{1}),\cdot \cdot \cdot ,\omega
(t_{k})).
\end{equation*}%
Here $\varphi _{k}(\omega (t_{1}),\cdot \cdot \cdot ,\omega (t_{k}))$ is a
bounded cylinder function of paths on $[0,T]$. We denote by $M^{0}(0,T)$ the
collection of all step processes.
\end{definition}

\begin{definition}
(Cylinder process of paths) \label{def3.1} A function $u(t,\omega
):[0,T]\times \Omega _{T}\rightarrow \mathbb{R}$ is called a cylinder path
process if there exists a time partition $\{t_{i}\}_{i=0}^{n}$ with $%
0=t_{0}<t_{1}<\cdot \cdot \cdot <t_{n}=T$, such that for each $k=0,1,\cdot
\cdot \cdot ,n-1$ and $t\in (t_{k},t_{k+1}]$
\begin{equation*}
u(t,\omega )=u_{k}(t,\omega (t);\omega (t_{1}),\cdot \cdot \cdot ,\omega
(t_{k})).
\end{equation*}%
Here for each $k$, the function $u_{k}:[t_{k},t_{k+1}]\times (\mathbb{R}%
^{d})^{(k+1)}\rightarrow \mathbb{R}$ is a $C^{\infty }$-function with
\begin{equation*}
u_{k}(t_{k},x;x_{1},\cdot \cdot \cdot
,x_{k-1},x)=u_{k-1}(t_{k},x;x_{1},\cdot \cdot \cdot ,x_{k-1})
\end{equation*}%
such that, all derivatives of $u_{k}$ have at most polynomial growth. We
denote by ${\mathcal{C}}^{\infty }(0,T)$ for the collection of all cylinder
path processes.
\end{definition}

The following proposition is easy.

\begin{proposition}
\label{prop3.2} Let $\eta,\zeta$ be step processes. Then
\begin{equation*}
u(t,\omega):=\int_{0}^{t}\eta(s,\omega)ds+\int_{0}^{t}\zeta(s,\omega)dB_{s}
\end{equation*}
belongs to ${\mathcal{C}}^{\infty}(0,T)$.
\end{proposition}

It is clear that ${\mathcal{C}}^{\infty}(0,T)\subset{\mathcal{C}}^{\infty
}(0,\bar{T})$ for $\bar{T}\geq T$. We also set
\begin{equation*}
{\mathcal{C}}^{\infty}(0,\infty):=\bigcup_{n=1}^{\infty}{\mathcal{C}}%
^{\infty }(0,n).
\end{equation*}
For $t\in \lbrack t_{k},t_{k+1})$, $n\in \mathbb{N}$, we denote
\begin{equation*}
D_{t}^{(n)}u(t,\omega):=\partial_{t+}^{(n)}u_{k}(t,x;x_{1},\cdot \cdot
\cdot,x_{k})|_{x=\omega(t),x_{1}=\omega(t_{1}),\cdot \cdot
\cdot,x_{k}=\omega(t_{k})}.
\end{equation*}
We denote $D_{t}=D_{t}^{(1)}$ for simplicity.

For $t\in (t_{k},t_{k+1}]$, we denote
\begin{align}
D_{x}u(t,\omega ):=& \partial _{x}u_{k}(t,x;x_{1},\cdot \cdot \cdot
,x_{k})|_{x=\omega (t),x_{1}=\omega (t_{1}),\cdot \cdot \cdot ,x_{k}=\omega
(t_{k})},  \label{Dx} \\
D_{x}^{2}u(t,\omega ):=& \partial _{x}^{2}u_{k}(t,x;x_{1},\cdot \cdot \cdot
,x_{k})|_{x=\omega (t),x_{1}=\omega (t_{1}),\cdot \cdot \cdot ,x_{k}=\omega
(t_{k})}, \\
\Delta _{x}u(t,\omega ):=& \mathrm{tr[}D_{x}^{2}u(t,\omega )\mathrm{]}.
\label{Lap}
\end{align}%
Let us indicate the relation of $D_{x}u$ with the well-known Malliavin
calculus. Let $\mathbb{D}$ be the Malliavin derivative operator. Then for
each $u\in \mathcal{C}^{\infty }(0,T)$, we have
\begin{equation*}
\mathbb{D}_{t}u_{t}(\omega )=D_{x}u(t,\omega ).
\end{equation*}%
But as in the above explanation, since the notion of $D_{x}u(t,\omega )$
corresponds much more like the classical derivative of $D_{x}u(t,x)$,
emphasizing simply the perturbation of state point than the Malliavin's one
emphasizing the perturbation of the whole path, thus we prefer to use the
denotation $D_{x}u(t,\omega )$. Another important convenience is that, with
this notation, the path PDEs discussed in this paper can be easily related
to the corresponding classical PDEs of parabolic types.

In fact, the above definition of derivatives corresponds perfectly with
Dupire's one, introduced originally in his deep insightful paper (2009) (see
also [CF2010]). An advantage of our new formulation in this paper is that we
do not need to define of our derivatives on a larger space of right
continuous paths with left limit. Our weak formulation is not only general
enough but also necessary to treat almost all existing results of stochastic
calculus, within the classical as well as within $G$-frameworks, into the
corresponding path dependent PDEs.

In the sequel, we shall give the definitions of $G$-Sobolev spaces. For
readers' convenience, we divide the discussions into two parts. First we
consider this problem in the framework of the classical Wiener probability
space, which presents a completely new point of view of It\^{o} processes.

In this paper, we fix a number $p>1$. For a step process $\eta \in
M^{0}(0,T) $, we set the norm
\begin{equation*}
\Vert \eta \Vert _{H_{G}^{p}}^{p}:=\mathbb{E}^{G}[\{ \int_{0}^{T}|\eta
_{s}|^{2}ds\}^{p/2}],\  \  \  \Vert \eta \Vert _{M_{G}^{p}}^{p}:=\mathbb{E}%
^{G}[\int_{0}^{T}|\eta _{s}|^{p}ds]\  \
\end{equation*}
and denote by $H_{G}^{p}(0,T)$ and $M_{G}^{p}(0,T)$ the completion of $%
M^{0}(0,T)$ with respect to the norms $\Vert \cdot \Vert _{H_{G}^{p}}$ and $%
\Vert \cdot \Vert _{M_{G}^{p}}$, respectively.

\section{Sobolev spaces on path space under Wiener expectation}

\subsection{ $P$-Sobolev spaces of path functions}

For the case $G(A)=\frac{1}{2}$tr$[A]$, the above $G$-expectation is just
the expectation $E_{P}$ of the Wiener probability $P$ and the $G$-Brownian
motion $B$ becomes the Wiener process. In the Wiener probability space we
denote the corresponding norms and spaces by
\begin{align*}
\Vert \cdot \Vert _{L_{P}^{p}}& :=\Vert \cdot \Vert _{L_{G}^{p}},\  \
L_{P}^{p}(\Omega _{T}):=L_{G}^{p}(\Omega _{T}),\  \  \\
\Vert \cdot \Vert _{H_{P}^{p}}& :=\Vert \cdot \Vert _{H_{G}^{p}},\  \
H_{P}^{p}(0,T):=H_{G}^{p}(0,T),\ \  \\
\Vert \cdot \Vert _{M_{P}^{p}}& :=\Vert \cdot \Vert _{M_{G}^{p}},\  \
M_{P}^{p}(0,T):=M_{G}^{p}(0,T).\
\end{align*}

In this case each cylinder process $u\in {\mathcal{C}}^{\infty }(0,\infty )$
has the following decomposition:

\begin{proposition}
\label{prop3.3} For each given $u\in{\mathcal{C}}^{\infty}(0,\infty)$ we
have
\begin{equation*}
u(t,\omega)=u(0,\omega)+\int_{0}^{t}\mathcal{A}u(s,\omega)ds+%
\int_{0}^{t}D_{x}u(s,\omega)dB_{s},
\end{equation*}
where
\begin{equation*}
\mathcal{A}u(s,\omega):=(D_{s}+\frac{1}{2}\Delta_{x})u(s,\omega)=D_{s}u(s,%
\omega)+\frac{1}{2}\Delta_{x}u(s,\omega).
\end{equation*}
\end{proposition}

\begin{definition}
\label{def3.4} 1) For a process $u\in {\mathcal{C}}^{\infty }(0,T)$, set
\begin{equation*}
\Vert u\Vert _{S_{P}^{p}}^{p}=E_{P}[\sup_{s\in \lbrack 0,T]}|u_{s}|^{p}].
\end{equation*}%
Denote by $S_{P}^{2}(0,T)$ the completion of $u\in {\mathcal{C}}^{\infty
}(0,T)$ w.r.t. the norm $\Vert \cdot \Vert _{S_{P}^{2}}$.

2) For $u\in {\mathcal{C}}^{\infty }(0,T)$, set
\begin{equation*}
\Vert u\Vert _{W_{\mathcal{A}}^{\frac{1}{2},1;p}}^{p}=E_{P}[\sup_{s\in \lbrack
0,T]}|u_{s}|^{p}+\int_{0}^{T}|{\mathcal{A}}u_{s}|^{p}ds+\{%
\int_{0}^{T}|D_{x}u_{s}|^{2}ds\}^{p/2}].
\end{equation*}
\end{definition}

\begin{proposition}
\label{prop3.5} The norm $\Vert \cdot \Vert _{W_{\mathcal{A}}^{\frac{1}{2},1;p}}$ is
closable in the space $S_{P}^{p}(0,T)$ in the following sense: Let $u^{n}\in
{\mathcal{C}}^{\infty }(0,T)$ be a Cauchy sequence w.r.t. the norm $\Vert
\cdot \Vert _{W_{\mathcal{A}}^{\frac{1}{2},1;p}}$. If $\Vert u^{n}\Vert
_{S_{P}^{p}}\rightarrow 0$, we have $\Vert u^{n}\Vert _{\mathcal{W}%
_{P}^{1,2;p}}\rightarrow 0$.
\end{proposition}

\begin{proof}
The proposition follows directly from the uniqueness of the decomposition
for the classical It\^{o} processes.
\end{proof}

\begin{definition}
We denote by $W_{\mathcal{A}}^{\frac{1}{2},1;p}(0,T)$ the completion of ${\mathcal{C}}%
^{\infty }(0,T)$ w.r.t. the norm $\Vert \cdot \Vert _{\mathcal{W}%
_{P}^{1,2;p}}$. From the above proposition, $W_{\mathcal{A}}^{\frac{1}{2},1;p}(0,T)$
can be considered as a subspace of $S_{P}^{p}(0,T)$. Now, the operators $%
D_{x}$ and $\mathcal{A}$, defined respectively in (\ref{Dx}) and Proposition %
\ref{prop3.3}, can be continuously extended to the space $W%
_{P}^{1,2;p}(0,T)$:
\begin{align*}
D_{x}& :\  W_{\mathcal{A}}^{\frac{1}{2},1;p}(0,T)\mapsto H_{P}^{p}(0,T), \\
\mathcal{A}& :\  W_{\mathcal{A}}^{\frac{1}{2},1;p}(0,T)\mapsto M_{P}^{p}(0,T).
\end{align*}
\end{definition}

The following proposition presents a completely new point of view for
classical It\^{o} processes.

\begin{proposition}
\label{prop3.6}Assume $u\in S_{P}^{p}(0,T)$. Then the following two
conditions are equivalent:

(i) $u\in W_{\mathcal{A}}^{\frac{1}{2},1;p}(0,T)$;

(ii) $u(t,\omega )=u(0,\omega )+\int_{0}^{t}\eta (s,\omega
)ds+\int_{0}^{t}v(s,\omega )dB_{s}$ with $\eta\in M_{P}^{p}(0,T)$ and $v\in H_{P}^{p}(0,T)$.

Moreover, we have
\begin{equation*}
(D_{t}+\frac{1}{2}\Delta_{x}{\normalsize )}u(t,\omega)=\eta(t,\omega ),\
D_{x}u(t,\omega)=v(t,\omega).
\end{equation*}
\end{proposition}

\begin{proof}
(i)$\implies $(ii) is obvious. Let's prove (ii)$\implies $(i). Choose step
processes $\eta ^{n},v^{n}$ such that $\Vert \eta ^{n}-\eta \Vert
_{M_{P}^{p}}\rightarrow 0$ and $\Vert v^{n}-v\Vert _{H_{P}^{p}}\rightarrow 0$%
. Set
\begin{equation*}
u^{n}(t,\omega ):=u(\omega _{0})+\int_{0}^{t}\eta ^{n}(s,\omega
)ds+\int_{0}^{t}v^{n}(s,\omega )dB_{s}.
\end{equation*}%
Clearly $u^{n}$ belongs to ${\mathcal{C}}^{\infty }(0,T)$ by Proposition \ref%
{prop3.2}. By Proposition \ref{prop3.3} and the uniqueness of the
decomposition for It\^{o} processes, we have
\begin{equation*}
(D_{t}+\frac{1}{2}\Delta _{x}{\normalsize )}u^{n}(t,\omega )=\eta
^{n}(t,\omega ),\ D_{x}u^{n}(t,\omega )=v^{n}(t,\omega ).
\end{equation*}%
So $u$ belongs to $W_{\mathcal{A}}^{\frac{1}{2},1;p}(0,T)$ with
\begin{equation*}
(D_{t}+\frac{1}{2}\Delta _{x}{\normalsize )}u(t,\omega )=\eta (t,\omega ),\
D_{x}u(t,\omega )=v(t,\omega ).
\end{equation*}
\end{proof}

\subsection{Backward SDEs in Wiener space and related PPDEs}

Recall that a classical backward SDE is defined on a Wiener probability
space $(\Omega ,\mathcal{F},P)$: to find a pair of processes $(Y,Z)\in
S_{P}^{p}(0,T)\times H_{P}^{p}(0,T)$ such that
\begin{equation}
Y_{t}=\xi +\int_{t}^{T}f(s,\omega ,Y_{s},Z_{s})ds-\int_{t}^{T}Z_{s}dB_{s},
\label{BSDE}
\end{equation}%
where $f:[0,T]\times \Omega \times \mathbb{R}^{n}\times \mathbb{R}^{n\times
d}\mapsto R^n$ is a given function and $\xi :\Omega
\mapsto \mathbb{R}^{n}$ is a given $\mathcal{F}_{T}$-measurable random
vector.

We shall show that the well-poseness of backward SDE (\ref{BSDE}) is
equivalent to that of the path dependent PDE: to find $u\in \mathcal{W}%
_{P}^{1,2;p}(0,T)$ such that
\begin{align}
(D_{t}+\frac{1}{2}\Delta _{x}{\normalsize )}u(t,\omega )+f(t,u(t,\omega
),D_{x}u(t,\omega ))& =0,\  \ t\in \lbrack 0,T),  \label{PPDE1} \\
u(T,\omega )& =\xi (\omega ).  \label{PPDE2}
\end{align}

\textbf{Assumption 1}. $f(t,\omega ,Y_{t},Z_{t})\in M_{P}^{p}(0,T)$ for any $%
(Y,Z)\in S_{P}^{p}(0,T)\times H_{P}^{p}(0,T)$.

\begin{theorem}
\label{thm4.1} Let $(Y,Z)$ be a solution to the backward SDE (\ref{BSDE}).
Then we have $u(t,\omega ):=Y_{t}(\omega )\in W_{\mathcal{A}}^{\frac{1}{2},1;p}(0,T)$
with $D_{x}u(t,\omega )=Z_{t}(\omega )$.

Moreover, given $u(t,\omega )\in W_{\mathcal{A}}^{\frac{1}{2},1;p}(0,T)$, the
following (i) and (ii) are equivalent:

(i) $(u, D_{x} u)$ is a solution to the backward SDE (\ref{BSDE});

(ii) $u$ is a solution to the path dependent PDE (\ref{PPDE1}-\ref{PPDE2}).%
\newline
\end{theorem}

\begin{remark}
By this theorem, we can directly apply the result of existence and
uniqueness of backward SDE to prove that of path dependent PDE (\ref{PPDE1}-%
\ref{PPDE2}). We recall the following existence and unqueness result
[Pardoux and Peng 1990] of the backward SDE (\ref{BSDE}) under the following
standard condition: $\xi \in L_{P}^{p}(\Omega _{T})$ and the function $f$
satisfying Lipschitz condition in $(y,z)$, namely, there exists a constant $%
C>0$, such that, for all $\omega \in \Omega $,
\begin{equation*}
|f(t,\omega ,y,z)-f(t,\omega ,y^{\prime },z^{\prime })|\leq C(|y-y^{\prime
}|+|z-z^{\prime }|),\  \ y,y^{\prime }\in \mathbb{R}^{n},\  \ z,z^{\prime }\in
\mathbb{R}^{n\times d}.
\end{equation*}%
This backward SDE can be directly read as a well-posed path dependent PDE (%
\ref{PPDE1}-\ref{PPDE2}).
\end{remark}

\begin{proof}
(i)$\implies $(ii). Assume that $(Y,Z)$ is a solution to the backward SDE (%
\ref{BSDE}). By Proposition \ref{prop3.6} we know that $u(t,\omega
):=Y_{t}(\omega )\in W_{\mathcal{A}}^{\frac{1}{2},1;p}(0,T)$ with $D_{x}u(t,\omega
)=Z_{t}(\omega )$ and
\begin{equation*}
{\mathcal{A}}u(t,\omega )+f(t,u(t,\omega ),D_{x}u(t,\omega ))=0.
\end{equation*}

(ii)$\implies $(i). Assume that $u(t,\omega )\in \mathcal{W}%
_{P}^{1,2;p}(0,T) $ is a solution to the path dependent PDE (\ref{PPDE1}-\ref%
{PPDE2}). By Proposition \ref{prop3.6} we have
\begin{align*}
u(t,\omega )& =u(0,\omega )+\int_{0}^{t}{\mathcal{A}}u(s,\omega
)ds+\int_{0}^{t}D_{x}u(s,\omega )dB_{s} \\
& =u(0,\omega )-\int_{0}^{t}f(s,u(s,\omega ),D_{x}u(s,\omega
))ds+\int_{0}^{t}D_{x}u(s,\omega )dB_{s} \\
& =\xi (\omega )+\int_{t}^{T}f(s,u(s,\omega ),D_{x}u(s,\omega
))ds-\int_{t}^{T}D_{x}u(s,\omega )dB_{s}.
\end{align*}
\end{proof}

\begin{remark}
An advantage of the above formulation is that the path dependent PDE can be
a system of PDEs, namely $u(t,\omega)$ can be $\mathbb{R}^{m}$-valued, or
even $H$-valued for a Hilbert space $H$.
\end{remark}

\begin{remark}
\label{rem4.2} Let us consider Markovian situations: $\xi=\varphi (B_{T}),\
f(t,\omega,y,z)=h(t,B_{t}(\omega),y,z)\ $for deterministic and continuous
functions$\  \varphi(x)\ $and\ $h(t,x,y,z)$ satisfying Lipschitz conditions
in $(y,z)$ and polynomial growth condition in $x$. Assume that $(Y,Z)$ is
the solution to the backward SDE (\ref{BSDE}) with the coefficients $%
(\varphi(B_{T})$, $h(t,B_{t},y,z))$. By the classical arguments in the BSDE
theory, we know that $Y$ is Markovian, i.e., there exists a deterministic
function $u(t,x)$ such that $Y_{t}=u(t,B_{t}).$ Assuming $u(t,x)$ is smooth,
we have
\begin{equation*}
\mathcal{A}u(t,\omega)=\partial_{t}u(t,B_{t})+\frac{1}{2}%
\Delta_{x}u(t,B_{t}),\ D_{x}u(t,\omega)=Du(t,B_{t}).
\end{equation*}
By Theorem \ref{thm4.1}, we have
\begin{equation*}
\partial_{t}u(t,B_{t})+\frac{1}{2}%
\Delta_{x}u(t,B_{t})+h(t,B_{t},u(t,B_{t}),Du(t,B_{t}))=0.
\end{equation*}
Equivalently,

\begin{equation*}
\partial _{t}u(t,x)+\frac{1}{2}\Delta _{x}u(t,x)+h(t,x,u(t,x),D_{x}u(t,x))=0.
\end{equation*}%
This is just the non-linear Feynman-Kac formula originally studied in Peng
(1991) and Pardoux-Peng (1992). In fact, with our new formulation, as
functions of $x$, $\varphi $ and $h$ only need to be measurable functions
with polynomial growth.
\end{remark}

\section{Sobolev spaces on path space under nonlinear expectation}

For the situation when $G$ is not a linear function the problem becomes more
subtle.

\subsection{$G$-Sobolev spaces of path functions}

In the $G$-expectation space, by $G$-It\^o's formula, for $u\in{%
\mathcal{C}}^{\infty}(0,\infty)$ we immediately obtain  the following decomposition.

\begin{proposition}
\label{prop3.7} For each given $u\in{\mathcal{C}}^{\infty}(0,\infty)$ we
have
\begin{align*}
u(t,\omega) &
=u(0,\omega)+\int_{0}^{t}D_{s}u(s,\omega)ds+\int_{0}^{t}D_{x}u(s,%
\omega)dB_{s}+\frac{1}{2}\int_{0}^{t}D_{x}^{2}u(s,\omega)d\langle
B\rangle_{s} \\
& =u(0,\omega)+\int_{0}^{t}\mathcal{A}_{G}u(s,\omega)ds+%
\int_{0}^{t}D_{x}u(s,\omega)dB_{s}+K_{t},
\end{align*}
where
\begin{equation*}
\mathcal{A}_{G}u(s,\omega):=(D_s+G\circ
D^2_x)u(s,\omega)=D_{s}u(s,\omega)+G(D_{x}^{2}u(s,\omega)),
\end{equation*}
and $K_{t}$ is a non-increasing $G$-martingale:%
\begin{equation*}
K_{t}:=\frac{1}{2}\int_{0}^{t}D_{x}^{2}u(s,\omega)d\langle
B\rangle_{s}-\int_{0}^{t}G(D_{x}^{2}u(s,\omega))ds.
\end{equation*}
\end{proposition}

\begin{definition}
\label{def3.8} 1) For $u\in {\mathcal{C}}^{\infty }(0,T)$, we set
\begin{equation*}
\Vert u\Vert _{S_{G}^{p}}^{p}=\mathbb{E}^{G}[\sup_{s\in \lbrack
0,T]}|u_{s}|^{p}].
\end{equation*}%
We denote by $S_{G}^{p}(0,T)$ the completion of $u\in {\mathcal{C}}^{\infty
}(0,T)$ w.r.t. the norm $\Vert \cdot \Vert _{S_{G}^{p}}$.

2) For $u\in {\mathcal{C}}^{\infty }(0,T)$, we set
\begin{equation*}
\Vert u\Vert _{W_{G}^{1,2;p}}^{p}=\mathbb{E}^{G}[\sup_{s\in \lbrack
0,T]}|u_{s}|^{p}+\int_{0}^{T}(|D_{s}u_{s}|^{p}+|D_{x}^{2}u_{s}|^{p})ds+\{%
\int_{0}^{T}|D_{x}u_{s}|^{2}ds\}^{p/2}].
\end{equation*}
\end{definition}

To define the $G$-Sobolev spaces, the key point is to show the uniqueness of
the decomposition for $G$-It\^o processes, which was actually solved by Song
(2012) in the one-dimensional $G$-expectation space and by Peng, Song and
Zhang (2012) for the multi-dimensional case.

For simplicity of notation, in the rest of this paper we only consider the
one-dimensional $G$-expectation space with $\bar{\sigma}^{2}:=\mathbb{E}%
^{G}[B_{1}^{2}]>\underline{\sigma}^{2}:=-\mathbb{E}^{G}[-B_{1}^{2}]$.

\begin{lemma}
\label{lem3.9} If
\begin{equation*}
u(t,\omega )=\int_{0}^{t}\zeta (s,\omega )ds+\int_{0}^{t}v(s,\omega )dB_{s}+%
\frac{1}{2}\int_{0}^{t}w(s,\omega )d\langle B\rangle _{s}=0,\ t\in \lbrack
0,T],
\end{equation*}%
with $\zeta ,w\in M_{G}^{p}(0,T)$ and $v\in H_{G}^{p}(0,T)$, then we have $\zeta =v=w=0$.
\end{lemma}

\begin{proof}
By the uniqueness of the decomposition for continuous semimartingales we
have $v=0$ and $\int_{0}^{t}\zeta(s,\omega)ds+\frac{1}{2}\int_{0}^{t}w(s,%
\omega)d\langle B\rangle_{s}=0$. By Corollary 3.5 in Song (2012) we conclude
that $\zeta=w=0$.
\end{proof}

\begin{proposition}
\label{prop3.10} The norm $\Vert \cdot \Vert _{W_{G}^{1,2;p}}$ is closable
in the space $S_{G}^{p}(0,T)$: Let $u^{n}\in {\mathcal{C}}^{\infty }(0,T)$
be a Cauchy sequence w.r.t. the norm $\Vert \cdot \Vert _{W_{G}^{1,2;p}}$.
If $\Vert u^{n}\Vert _{S_{G}^{p}}\rightarrow 0$, we have $\Vert u^{n}\Vert
_{W_{G}^{1,2;p}}\rightarrow 0$.
\end{proposition}

\begin{proof}
The proposition follows directly from the uniqueness of the decomposition
for $G$-It\^o processes.
\end{proof}

Denote by $W_{G}^{1,2;p}(0,T)$ the completion of ${\mathcal{C}}^{\infty
}(0,T)$ w.r.t. the norm $\Vert \cdot \Vert _{W_{G}^{1,2;p}}$. By the above
proposition, $W_{G}^{1,2;p}(0,T)$ can be considered as a subspace of $%
S_{G}^{p}(0,T)$. Now the differential operators $D_{t}$, $D_{x}$ and $%
D_{x}^{2}$ defined on ${\mathcal{C}}^{\infty }(0,T)$ can be continuously
extended to the space $W_{G}^{1,2;p}(0,T)$.

\begin{proposition}
\label{prop3.11} Assume $u\in S_{G}^{p}(0,T)$. Then the following two
conditions are equivalent:

(i) $u\in W_{G}^{1,2;p}(0,T)$;

(ii) $u$ is of the form:
\begin{equation*}
u(t,\omega )=u(0,\omega )+\int_{0}^{t}\zeta (s,\omega
)ds+\int_{0}^{t}v(s,\omega )dB_{s}+\frac{1}{2}\int_{0}^{t}w(s,\omega
)d\langle B\rangle _{s},
\end{equation*} where $ \zeta,w\in M_{G}^{p}(0,T)$ and $ v\in H_{G}^{p}(0,T)$

Moreover, we have
\begin{equation*}
D_{t}u(t,\omega )=\zeta (t,\omega ),\ D_{x}u(t,\omega )=v(t,\omega ),\
D_{x}^{2}u(t,\omega )=w(t,\omega ).
\end{equation*}
\end{proposition}

\begin{proof}
(i) $\implies $ (ii) is obvious. Let's prove (ii) $\implies $ (i). It
suffices to prove it for the case that $\zeta ,v,w$ are step processes. Set $%
t_{k}^{n}=\frac{kT}{2^{n}}$ and
\begin{equation*}
Q^{n}(t,\omega ):=\sum_{k=0}^{2^{n}-1}(B_{t_{k+1}^{n}\wedge
t}-B_{t_{k}^{n}\wedge t})^{2}=\int_{0}^{t}\lambda ^{n}(s,\omega
)dB_{s}+\langle B\rangle _{t},
\end{equation*}%
where $\lambda ^{n}(t,\omega
)=\sum_{k=0}^{2^{n}-1}2(B_{t}-B_{t_{k}})1_{]t_{k},t_{k+1}]}(t)$. Set
\begin{align*}
u^{n}(t,\omega )& :=u(\omega _{0})+\int_{0}^{t}\zeta (s,\omega
)ds+\int_{0}^{t}v(s,\omega )dB_{s}+\frac{1}{2}\int_{0}^{t}w_{s}dQ^{n}(s,%
\omega ) \\
& =u(\omega _{0})+\int_{0}^{t}\zeta (s,\omega )ds+\int_{0}^{t}(v(s,\omega )+%
\frac{1}{2}w(s,\omega )\lambda ^{n}(s,\omega ))dB_{s}+\int_{0}^{t}\frac{1}{2}%
w(s,\omega )d\langle B\rangle _{s}.
\end{align*}%
Clearly $u^{n}$ belongs to ${\mathcal{C}}^{\infty }(0,T)$. By Proposition %
\ref{prop3.7} and the uniqueness of the decomposition for $G$-It\^{o}
processes, we have
\begin{equation*}
D_{t}u^{n}(t,\omega )=\zeta (t,\omega ),\ D_{x}u^{n}(t,\omega )=v(t,\omega )+%
\frac{1}{2}w(t,\omega )\lambda ^{n}(t,\omega ),\ D_{x}^{2}u^{n}(t,\omega
)=w(t,\omega ).
\end{equation*}%
It's easy to check that $\mathbb{E}^{G}[(\int_{0}^{T}|D_{x}u^{n}(t,\omega
)-v(t,\omega )|^{2}dt)^{p/2}]\rightarrow 0$. So $u$ belongs to $W_{G}^{1,2;p}(0,T)$
with
\begin{equation*}
D_{t}u(t,\omega )=\zeta (t,\omega ),\ D_{x}u(t,\omega )=v(t,\omega ),\
D_{x}^{2}u(t,\omega )=w(t,\omega ).
\end{equation*}
\end{proof}

\begin{proposition}
For each $u,v\in W_{G}^{1,2;p}(0,T)$, we have $c_{1}u+c_{2}v\in
W_{G}^{1,2;p}(0,T)$ and
\begin{eqnarray*}
D_{t}(c_{1}u+c_{2}v) &=&c_{1}D_{t}u+c_{2}D_{t}v,\  \\
D_{x}(c_{1}u+c_{2}v) &=&c_{1}D_{x}u+c_{2}D_{x}v,\  \
D_{x}^{2}(c_{1}u+c_{2}v)=c_{1}D_{x}^{2}u+c_{2}D_{x}^{2}v.
\end{eqnarray*}%
Moreover, if their product $uv$ is also in $W_{G}^{1,2;p}(0,T)$, then
\begin{eqnarray*}
D_{t}(uv) &=&vD_{t}u+uD_{t}v,\  \ D_{x}(uv)=vD_{x}u+uD_{x}v \\
D_{x}^{2}(uv) &=&vD_{x}^{2}u+uD_{x}^{2}v+2D_{x}uD_{x}v.
\end{eqnarray*}
\end{proposition}

The proof is simply to apply the above proposition combined with the It\^{o}'s
formula for $uv$.

\begin{remark}
\label{rem3.12} 1) By Proposition \ref{prop3.11} we note that the equality
\begin{equation*}
D_{x}^{2}u(t,\omega )=D_{x}(D_{x}u)(t,\omega )
\end{equation*}%
does NOT hold for general $u\in W_{G}^{1,2;p}(0,T)$ although it holds for $%
u\in {\mathcal{C}}^{\infty }(0,T)$. Let's see how this happens from a simple
example: Let $u(t,\omega )=\langle B\rangle _{t}$, $t\in \lbrack 0,1]$. By
the definition we have
\begin{equation*}
D_{t}u(t,\omega )=D_{x}u(t,\omega )=0,\,\,\,\, D_{x}^{2}u(t,\omega )=2.
\end{equation*}%
Set $t_{k}^{n}=\frac{k}{2^{n}}$ and $u^{n}(t,\omega
)=\sum_{k=0}^{2^{n}-1}(B_{t_{k+1}^{n}\wedge t}-B_{t_{k}^{n}\wedge t})^{2}$.
By the definition we have
\begin{equation*}
D_{t}u^{n}(t,\omega )=0,\,\,\,\, D_{x}u^{n}(t,\omega
)=\sum_{k}2(B_{t}-B_{t_{k}})1_{]t_{k},t_{k+1}]}(t),\,\,\,\,D_{x}^{2}u^{n}(t,\omega
)=2.
\end{equation*}%
It is easily seen that $u^{n}\rightarrow u$ in $W_{G}^{1,2;p}(0,T)$.
Particularly, $D_{x}u^{n}\rightarrow D_{x}u$ in $H_{G}^{p}(0,T)$. However,
\begin{equation*}
D_{x}(D_{x}u^{n})(t,\omega )=D_{x}^{2}u^{n}(t,\omega )=2
\end{equation*}%
does NOT converge to
\begin{equation*}
D_{x}(D_{x}u)(t,\omega )=0.
\end{equation*}

2) Compared to Proposition \ref{prop3.6}, here the derivatives $D_{t}u,\
D_{x}u,\ D_{x}^{2}u$ can be distinguished clearly.
\end{remark}

\subsection{ Backward SDEs driven by $G$-Brownian motion}

In this section we show that a backward stochastic differential equation \
is in fact a path dependent PDE

\subsubsection{One-one correspondence}

Let us consider backward SDEs driven by $G$-Brownian motion in the following
from: to find $Y\in S_{G}^{p}(0,T)$, $Z\in H_{G}^{p}(0,T),\eta \in
M_{G}^{p}(0,T)$ such that
\begin{equation}
Y_{t}=\xi +\int_{t}^{T}f(s,Y_{s},Z_{s},\eta
_{s})ds-\int_{t}^{T}Z_{s}dB_{s}-(K_{T}-K_{t}),  \label{GBSDE}
\end{equation}%
where $K_{t}=\frac{1}{2}\int_{0}^{t}\eta _{s}d\langle B\rangle
_{s}-\int_{0}^{t}G(\eta _{s})ds$, $f:{\normalsize [0,T]\times }\mathbb{%
R\times R}^{d}\times \mathbb{S}{\normalsize (d)\mapsto }\mathbb{R}$ is a
given function and $\xi \in L_{G}^{p}(\Omega _{T})$ is a given random
variable.

The related problem of path dependent PDEs: to find $u\in W_{G}^{1,2;p}(0,T)$
such that
\begin{align}
D_{t}u+G(D_{x}^{2}u)+f(t,u,D_{x}u,D_{x}^{2}u)& =0,\  \ t\in \lbrack 0,T),
\label{PPDE3} \\
u(T,\omega )& =\xi (\omega ).  \label{PPDE4}
\end{align}%
We call $u$ a $W_{G}^{1,2;p}$-solution of the path dependent PDE (\ref{PPDE3}%
-\ref{PPDE4}).

\textbf{Assumption 2.}

$f(t,\omega ,Y_{t},Z_{t},\eta _{t})\in M_{G}^{p}(0,T)$ for any $(Y,Z,\eta
)\in S_{G}^{p}(0,T)\times H_{G}^{p}(0,T)\times M_{G}^{p}(0,T)$.

\begin{theorem}
\label{thm4.3} Let $(Y,Z,\eta )$ be a solution to the backward SDE (\ref%
{GBSDE}). Then we have $u(t,\omega ):=Y_{t}(\omega )\in W_{G}^{1,2;p}(0,T)$
with $D_{x}u(t,\omega )=Z_{t}(\omega )$ and $D_{x}^{2}u(t,\omega )=\eta
_{t}(\omega )$.

Moreover, for $u(t,\omega )\in W_{G}^{1,2;p}(0,T)$, the following conditions
are equivalent:

(i) $(u, D_{x} u, D^{2}_{x} u)$ is a solution to the backward SDE (\ref%
{GBSDE});

(ii) $u$ is a $W_{G}^{1,2;p}$-solution to the path dependent PDE (\ref{PPDE3}%
-\ref{PPDE4}).
\end{theorem}

\subsubsection{Solutions of path dependent PDEs defined by $G$-BSDEs}

Now let's consider a special case of the path dependent PDE (\ref{PPDE3}-\ref%
{PPDE4}): \emph{$f$ is independent of $D^{2}_{x}u$.}
\begin{align}
D_{t}u+ G(D^{2}_{x} u)+f(t,u,D_{x} u) & =0, \  \ t\in[0,T),  \label{PPDE3'} \\
u(T,\omega) & =\xi(\omega).  \label{PPDE4'}
\end{align}

Let $u\in W_{G}^{1,2;p}(0,T)$ be a solution to the path dependent PDE (\ref%
{PPDE3'}-\ref{PPDE4'}). By Theorem \ref{thm4.3}, the processes
\begin{equation*}
Y_{t}:=u(t,\omega ),\,\,\, Z_{t}:=D_{x}u(t,\omega ),\,\,\, K_{t}:=\frac{1}{2}%
\int_{0}^{t}D_{x}^{2}u(s,\omega )d\langle B\rangle
_{s}-\int_{0}^{t}G(D_{x}^{2}u(s,\omega ))ds
\end{equation*}%
satisfy the following backward SDE:
\begin{equation}
Y_{t}=\xi
+\int_{t}^{T}f(s,Y_{s},Z_{s})ds-\int_{t}^{T}Z_{s}dB_{s}-(K_{T}-K_{t}),
\label{GBSDE-s4}
\end{equation}%
which is a type of $G$-BSDE studied in [HJPS2012] (see the Appendix).

Let $(Y,Z,K)$ be a solution of backward SDE (\ref{GBSDE-s4}). Generally, we
don't know whether $Y\in W_{G}^{1,2;p}(0,T)$. However, $u(t,\omega ):=Y_{t}$
is still a reasonable candidate for the solution of the path dependent PDE (%
\ref{PPDE3'}-\ref{PPDE4'}). In section 5, we shall formulate it as a weak solution by
 introducing the Sobolev space $W_{\mathcal{A}_G}^{\frac{1}{2},1;p}(0,T)$, which,
 corresponding to $W_{\mathcal{A}}^{\frac{1}{2},1;p}(0,T)$ in the classical Wiener
 probability space, is an expansion of  $W_{G}^{1,2;p}(0,T)$ with weaker derivatives.
\subsubsection{Examples and applications}

\begin{example}
Let $\eta \in M_{G}^{p}(0,T)$. To find $u\in W_{G}^{1,2;p}(0,T)$ such that
\begin{align}
D_{t}u+G(D_{x}^{2}u+\eta _{t})& =0  \label{e1.1} \\
u(T, \omega)& =0.  \label{e1.2}
\end{align}

Assume that $u\in W_{G}^{1,2;p}(0,T)$ is a solution to (\ref{e1.1}-\ref{e1.2}%
). Then
\begin{equation*}
u(t,\omega )=\mathbb{E}_{t}^{G}[\frac{1}{2}\int_{t}^{T}\eta _{s}d\langle
B\rangle _{s}].
\end{equation*}

In fact,
\begin{align*}
u_{t} & =-\int_{t}^{T}D_{s}u_{s}ds-\int_{t}^{T}D_{x}u_{s}dB_{s}-\frac{1}{2}%
\int_{t}^{T}D^{2}_{x}u_{s}d\langle B\rangle_{s} \\
& = -(M_{T}-M_{t}) + \frac{1}{2}\int_{t}^{T}\eta_{s}d\langle B\rangle_{s},
\end{align*}
where $M_{t}:=\int_{0}^{t}D_{x}u_{s}dB_{s}+\frac{1}{2}%
\int_{0}^{t}(D^{2}_{x}u_{s}+\eta_{s})d\langle B\rangle_{s}
-\int_{0}^{t}G(D^{2}_{x}u_{s}+\eta_{s})ds$ is a $G$-martingale. So
\begin{equation*}
u_{t}=\mathbb{E}^{G}_{t}[\frac{1}{2}\int_{t}^{T}\eta_{s}d\langle B\rangle
_{s}].
\end{equation*}
\end{example}

\begin{example}
Let $\eta \in M_{G}^{p}(0,T)$. To find $v\in W_{G}^{1,2;p}(0,T)$ such that
\begin{align}
D_{t}v+G^{\eta }(D_{x}^{2}v)& =0  \label{e2.1} \\
v(T, \omega)& =0,  \label{e2.2}
\end{align}%
where $G^{\eta }(\zeta _{s})=\frac{1}{2}[G(\zeta _{s}+\eta _{s})+G(\zeta
_{s}-\eta _{s})]$.

Assume that $v\in W_{G}^{1,2;p}(0,T)$ is a solution to (\ref{e2.1}- \ref%
{e2.2}). Then
\begin{equation*}
v_{t}=\limsup_{n\rightarrow \infty }\mathbb{E}_{t}^{G}[\frac{1}{2}%
\int_{t}^{T}\delta _{n}(s)\eta _{s}d\langle B\rangle _{s}],
\end{equation*}%
where $\delta _{n}(s)=\Sigma _{i=0}^{n-1}(-1)^{i}1_{]\frac{iT}{n},\frac{%
(i+1)T}{n}]}$.

Actually,
\begin{align*}
v_{t} & =-\int_{t}^{T}D_{s}v_{s}ds-\int_{t}^{T}D_{x}v_{s}dB_{s}-\frac{1}{2}%
\int_{t}^{T}D^{2}_{x}v_{s}d\langle B\rangle_{s} \\
& = -\int_{t}^{T}D_{x}v_{s}dB_{s}-\frac{1}{2}\int_{t}^{T}(D^{2}_{x}v_{s}+%
\delta_{n}(s)\eta_{s})d\langle B\rangle_{s}
+\int_{t}^{T}G^{\eta}(D^{2}_{x}v_{s})ds \\
& + \frac{1}{2}\int_{t}^{T}\delta_{n}(s)\eta_{s}d\langle B\rangle_{s}.
\end{align*}
So
\begin{align*}
& v_{t}+\limsup_{n\rightarrow \infty} \mathbb{E}^{G}_{t}[\frac{1}{2}\int
_{t}^{T}(D^{2}_{x}v_{s}+\delta_{n}(s)\eta_{s})d\langle B\rangle_{s} +\int
_{t}^{T}G^{\eta}(D^{2}_{x}v_{s})ds] \\
& = \frac{1}{2}\int_{t}^{T}\delta_{n}(s)\eta_{s}d\langle B\rangle_{s}.
\end{align*}
Noting that
\begin{equation*}
\limsup_{n\rightarrow \infty} \mathbb{E}^{G}_{t}[\frac{1}{2}%
\int_{t}^{T}(D^{2}_{x}v_{s}+\delta_{n}(s)\eta_{s})d\langle B\rangle_{s}
+\int_{t}^{T}G^{\eta}(D^{2}_{x}v_{s})ds]=0,
\end{equation*}
we get
\begin{equation*}
v_{t}=\limsup_{n\rightarrow \infty} \mathbb{E}^{G}_{t}[\frac{1}{2}%
\int_{t}^{T}\delta_{n}(s)\eta_{s}d\langle B\rangle_{s}].
\end{equation*}
\end{example}

\begin{example}
Let $\eta \in M_{G}^{p}(0,T)$ and $\varepsilon \in \lbrack 0,\frac{\overline{%
\sigma }^{2}-\underline{\sigma }^{2}}{2}]$. To find $u\in W_{G}^{1,2;p}(0,T)$
such that
\begin{align}
D_{t}u+G_{\varepsilon }(D_{x}^{2}u)+\frac{1}{2}\eta _{t}& =0  \label{e3.1} \\
u(T, \omega)& =0,  \label{e3.2}
\end{align}%
where $G_{\varepsilon }(a)=\frac{1}{2}[(\overline{\sigma }^{2}-\varepsilon
)a^{+}-(\underline{\sigma }^{2}+\varepsilon )a^{-}]$.

Assume that $u\in W_{G}^{1,2;p}(0,T)$ is a solution to (\ref{e3.1}- \ref%
{e3.2}). Then
\begin{equation*}
u(t,\omega )=\mathbb{E}_{t}^{G_{\varepsilon }}[\frac{1}{2}\int_{t}^{T}\eta
_{s}ds].
\end{equation*}

In fact,
\begin{align*}
u_{t} & =-\int_{t}^{T}D_{s}u_{s}ds-\int_{t}^{T}D_{x}u_{s}dB_{s}-\frac{1}{2}%
\int_{t}^{T}D^{2}_{x}u_{s}d\langle B\rangle_{s} \\
& = -(M^{\varepsilon}_{T}-M^{\varepsilon}_{t}) + \frac{1}{2}%
\int_{t}^{T}\eta_{s}ds,
\end{align*}
where $M^{\varepsilon}_{t}:=\int_{0}^{t}D_{x}u_{s}dB_{s}+\frac{1}{2}\int
_{0}^{t}D^{2}_{x}u_{s}d\langle B\rangle_{s}
-\int_{0}^{t}G_{\varepsilon}(D^{2}_{x}u_{s})ds$ is a $G_{\varepsilon}$%
-martingale. So
\begin{equation*}
u_{t}=\mathbb{E}^{G_{\varepsilon}}_{t}[\frac{1}{2}\int_{t}^{T}\eta_{s}ds].
\end{equation*}
\end{example}

Set $\beta=\frac{\overline{\sigma}^{2}}{\underline{\sigma}^{2}}$ and $\gamma=%
\frac{\beta-1}{\beta+1}$. For any $a, \alpha \in R$ and $\varepsilon \in[0,
\frac{\overline{\sigma}^{2}-\underline{\sigma}^{2}}{2}]$, it's easy to check
that
\begin{equation*}
G(a+\gamma|\alpha|)\geq G^{\alpha}(a)\geq G_{\varepsilon}(a)+\frac{1}{2}%
\varepsilon|\alpha|.
\end{equation*}
So by the comparison theorem for the (path dependent) PDEs, we recover the
estimates obtained in [Song12] and [PSZ12].

\begin{corollary}
For any $\eta \in M^{1}_{G}(0,T)$, we have
\begin{equation*}
\gamma \mathbb{E}^{G}[\int_{0}^{T}|\eta_{s}|d\langle B\rangle_{s}] \geq
\limsup_{n\rightarrow \infty}\mathbb{E}^{G}[\int_{0}^{T}\delta_{n}(s)\eta
_{s}d\langle B\rangle_{s}] \geq \varepsilon \mathbb{E}^{G_{\varepsilon}}[%
\int_{0}^{T}|\eta_{s}|ds].
\end{equation*}
\end{corollary}

\section {Weak solutions in $W_{\mathcal{A}_G}^{\frac{1}{2},1;p}(0,T)$}

\subsection{ $G$-Sobolev spaces $W_{\mathcal{A}_G}^{\frac{1}{2},1;p}(0,T)$}

\subsection{Definition of $G$-Sobolev spaces}

For each $u\in W_{G}^{1,2;p}(0,T)$ with $D_{t}u=\lambda $, $D_{x}u=\zeta $
and $D_{x}^{2}u=\gamma $, we set
\begin{equation*}
u(t,\omega )=u_{0}+\int_{0}^{t}\mathcal{A}_{G}u(s,\omega
)ds+\int_{0}^{t}\zeta (s,\omega )dB_{s}+K_{t}^{\gamma }
\end{equation*}%
where we denote
\begin{equation*}
K_{t}^{\gamma }=\frac{1}{2}\int_{0}^{t}\gamma (s,\omega )d\left \langle
B\right \rangle _{s}-\int_{0}^{t}G(\gamma (s,\omega ))d\left \langle B\right
\rangle _{s}.
\end{equation*}%
and
\begin{equation*}
\mathcal{A}_{G}u=\lambda +G(D_{x}^{2}u).
\end{equation*}

For $u,v\in W_{G}^{1,2;p}(0,T)$, set
\begin{equation*}
d_{W_{\mathcal{A}_G}^{\frac{1}{2},1;p}}(u,v)=\mathbb{E}^{G}[\sup_{s\in \lbrack
0,T]}|u_{s}-v_{s}|^{p}+( \int_{0}^{T}|D_{x}(u_{s}-v_{s})|^{2}ds)^{p/2}+ \int_{0}^{T}|\mathcal{A}%
_{G}u_s-\mathcal{A}_{G}v_s|^{p}ds]^{1/p}.
\end{equation*}

For each fixed $\eta \in M_{G}^{p}(0,T)$ we set
\begin{equation*}
W_{G}^{1,2;p}(\eta):=\{u\in W_{G}^{1,2;p}(0,T):\mathcal{A}_{G}(u)+\eta
=0\},
\end{equation*}%
which is the collection of all solutions of the PDE $\mathcal{A}_{G}(u)+\eta
=0$ in $W_{G}^{1,2;p}(0,T)$. We denote by $W_{\mathcal{A}_G}^{\frac{1}{2},1;p}(\eta )$
the completion of $W_{G}^{1,2;p}(\eta )$ under
$d_{W_{\mathcal{A}_G}^{\frac{1}{2},1;p}}$.

 We denote by $%
W_{\mathcal{A}_G}^{\frac{1}{2},1;p}(0,T)$ the collection of all processes $u\in
S_{G}^{p}(0,T)$ with the following property: there exists a Cauchy sequence $\{u^{n}\} \subset
W_{G}^{1,2;p}(0,T)$ with respect to the metric $d_{W_{\mathcal{A}_G}^{\frac{1}{2},1;p}}$
such that $\Vert u^{n}-u\Vert _{S_{G}^{p}}\rightarrow 0.$

Denote by $\mathcal{K}^p$ the closure of $\mathcal{K}^{0}:=\{K_{\cdot
}^{\gamma }:\gamma \in M_{G}^{p}(0,T)\}$ in the space $S_{G}^{p}(0,T)$.
Obviously, we have
\begin{equation*}
W_{\mathcal{A}_G}^{\frac{1}{2},1;p}(0,T)=\{u=u_{0}+\int_{0}^{t}\beta (s)ds+\int_{0}^{t}\zeta (s)dB_{s}+K_{t}:\  \beta \in
M_{G}^{p}(0,T)\text{, }\zeta \in
H_{G}^{p}(0,T)\text{, }K\in \mathcal{K}^p\},
\end{equation*}%
and
\begin{equation*}
W_{\mathcal{A}_G}^{\frac{1}{2},1;p}(\eta)=\{u=u_{0}+\int_{0}^{t}\eta (s
)ds+\int_{0}^{t}\zeta (s )dB_{s}+K_{t}:\  \zeta \in H_{G}^{p}(0,T)%
\text{, }K\in \mathcal{K}^p\}.
\end{equation*}

\subsubsection{ A review of the structure of $G$-martingales}

In order to understand the spaces $W_{\mathcal{A}_G}^{\frac{1}{2},1;p}(0,T)$ and $%
W_{\mathcal{A}_G}^{\frac{1}{2},1;p}(\eta)$, we recall the structure of $G$-martingales.
\cite{P07b} proved that for any $\xi \in C^{\infty }(\Omega _{T})$, the $G$%
-martingale $X_{t}=\mathbb{E}_{t}^{G}[\xi ]$ has the following
representation:
\begin{equation}
X_{t}=\mathbb{E}^{G}[\xi ]+\int_{0}^{t}Z_{s}dB_{s}+\frac{1}{2}%
\int_{0}^{t}\eta _{s}d\langle B\rangle _{s}-\int_{0}^{t}G(\eta _{s})ds
\label{MR}
\end{equation}%
for some $Z\in H_{G}^{p}(0,T),\eta \in M_{G}^{p}(0,T)$ and conjectured that for any $\xi \in
L_{G}^{p}(\Omega _{T})$ the representation (\ref{MR}) holds. Besides, \cite%
{P07b} showed that for any $\eta \in M_{G}^{p}(0,T)$,
\begin{equation*}
K_{t}:=\frac{1}{2}\int_{0}^{t}\eta _{s}d\langle B\rangle
_{s}-\int_{0}^{t}G(\eta _{s})ds
\end{equation*}%
is a non-increasing $G$-martingale. So any process $K\in \mathcal{K}^p$ is a
non-increasing $G$-martingale with $K_{T}\in L_{G}^{p}(\Omega _{T})$. By
Theorem 5.4 in \cite{Song11b}, the converse statement is also right.

For $p\geq 1$ and $\xi \in C^{\infty }(\Omega _{T})$, set $\Vert \xi \Vert _{%
\mathbb{L}_{G}^{p}}^{p}=\mathbb{E}^{G}[\sup_{t\in \lbrack 0,T]}|\mathbb{E}%
_{t}^{G}[\xi ]|^{p}]$. Denote by $\mathbb{L}_{G}^{p}(\Omega _{T})$ the
closure of $C^{\infty }(\Omega _{T})$ with respect to the norm $\Vert \cdot
\Vert _{\mathbb{L}_{G}^{p}}$ in $L_{G}^{p}(\Omega _{T})$. \cite{STZ09}
showed that for any $\xi \in \mathbb{L}_{G}^{2}(\Omega _{T})$ the $G$-martingale $%
X_{t}:=\mathbb{E}_{t}^{G}[\xi ]$ has the following decomposition:
\begin{equation}
X_{t}=\mathbb{E}^{G}[\xi ]+\int_{0}^{t}Z_{s}dB_{s}+K_{t},  \label{MD}
\end{equation}%
where $K_{t}$ is a non-increasing $G$-martingale.

Song \cite{Song11} showed that $\mathbb{L}_{G}^{p}(\Omega _{T})\supset {%
L_{G}^{q}}(\Omega _{T})$ for any $1\leq p<q$. Moreover, \cite{Song11} proved
that the decomposition (\ref{MD}) holds for any $\xi \in L_{G}^{p}(\Omega
_{T})$ with $p>1$. Independently, \cite{STZ} showed that $\mathbb{L}_{G}^{2}(\Omega
_{T})\supset {L_{G}^{q}}(\Omega _{T})$ for any $q>2$.

Let $u\in S_{G}^{p}(0,T)$. We say that $u$ is \emph{a weak $G$-It\^{o}
process} if there exist $\eta\in M_{G}^{p}(0,T)$ and $\zeta \in H_{G}^{p}(0,T)$ such that
\begin{equation*}
u(t,\omega )-\int_{0}^{t}\eta _{s}ds-\int_{0}^{t}\zeta _{s}dB_{s}
\end{equation*}%
is a non-increasing $G$-martingale. Clearly, $\zeta $ is determined uniquely
by $u$. We denote $\zeta $ by $D_{x}u$, which is consistent with the
definition of the operator $D_{x}$ for $u\in W_{G}^{1,2;p}(0,T)$.

\begin{proposition}
$W_{\mathcal{A}_G}^{\frac{1}{2},1;p}(0,T)$ is exactly the totality of all weak $G$-It%
\^{o} processes.
\end{proposition}

The following proposition provides closability of the metric
$d_{W_{\mathcal{A}_G}^{\frac{1}{2},1;p}}$
in $W_{G}^{1,2;p}(\eta)$.

\begin{proposition}
\label{prop5.2} The metric $d_{W_{\mathcal{A}_G}^{\frac{1}{2},1;p}}$ is closable in $%
W_{G}^{1,2;p}(\eta)$ under the distance of $S_{G}^{p}(0,T)$:
Let $\{u^{n}\}_{n=1}^{\infty }$ and $\{ \bar{u}^{n}\}_{n=1}^{\infty }$ be
two Cauchy sequences in $%
W_{G}^{1,2;p}(\eta)$ w.r.t. the metric $d_{%
W_{\mathcal{A}_G}^{\frac{1}{2},1;p}}$. If $\Vert u^{n}-\bar{u}^{n}\Vert
_{S_{G}^{p}}\rightarrow 0$, we have $d_{W_{\mathcal{A}_G}^{\frac{1}{2},1;p}}(u^{n},\bar{u%
}^{n})\rightarrow 0$.
\end{proposition}

\begin{proof}
The Cauchy limits of $u^{n},\bar{u}^{n}\in W_{G}^{1,2;p}(0,T)$ are
denoted by
\begin{equation*}
u(t,\omega )=u(0,\omega )-\int_{0}^{t}\eta (s,\omega
)ds+\int_{0}^{t}v(s,\omega )dB_{s}+K_{t}
\end{equation*}%
and%
\begin{equation*}
\bar{u}(t,\omega )=\bar{u}(0,\omega )-\int_{0}^{t}\eta (s,\omega
)ds+\int_{0}^{t}\bar{v}(s,\omega )dB_{s}+\bar{K}_{t}
\end{equation*}%
respectively, with $u(t,\omega )\equiv \bar{u}(t,\omega )$. Thus $%
\int_{0}^{t}v(s,\omega )dB_{s}+K_{t}\equiv \int_{0}^{t}\bar{v}(s,\omega )dB_{s}+\bar{K}_{t}$. It follows from the uniqueness of
decomposition theorem of $G$-martingale that $v(t,\omega )\equiv \bar{v}(t,\omega )$ and $K_{t}\equiv \bar{K}_{t}$.
\end{proof}

\begin{proposition}
\label{prop5.3} Let $\eta \in M_{G}^{p}(0,T)$ be given. Then
$u\in W_{\mathcal{A}_G}^{\frac{1}{2},1;p}(\eta )$ if and only if $u\in W_{\mathcal{A}_G}^{\frac{1}{2},1;p}(0,T)$
and $u_{t}-\int_{0}^{t}D_{x}u(s)dB_{s}-\int_{0}^{t}\eta _{s}ds$ is a
non-increasing $G$-martingale.
\end{proposition}

\subsection{Fully nonlinear path dependent PDEs}

Let's formulate the weak solution to the path dependent PDE (\ref{PPDE3'}-%
\ref{PPDE4'}) in the $G$-Sobolev space $W_{\mathcal{A}_G}^{\frac{1}{2},1;p}(0,T)$.

\begin{definition}
We say $u\in W_{\mathcal{A}_G}^{\frac{1}{2},1;p}(0,T)$ is a $W_{\mathcal{A}_G}^{\frac{1}{2},1;p}$-solution to the
path dependent PDE (\ref{PPDE3'}-\ref{PPDE4'}) if
\begin{equation*}
u(T,\omega )=\xi (\omega ),\,\, u\in W_{\mathcal{A}_G}^{\frac{1}{2},1;p}(\eta )\,\, \textrm{ with } \,\,
\eta _{t}(\omega )=g(t,\omega ,u(t,\omega ),D_{x}u(t,\omega )).
\end{equation*}
\end{definition}

The following theorem says that the  $W_{\mathcal{A}_G}^{\frac{1}{2},1;p}$-solution of the
PPDE (\ref{PPDE3'}-\ref{PPDE4'}) corresponds exactly to the solution of $G$-BSDE (\ref%
{GBSDE-s4}) studied in [HJPS2012]. 

\begin{theorem}
\label{thm5.5}(i) Assume $(Y,Z,K)$ is a solution to the backward SDE (\ref%
{GBSDE-s4}) and $g(t,\omega ,Y_{t},Z_{t})\in M_{G}^{p}(0,T)$.

Then we have $u(t,\omega ):=Y_{t}(\omega )\in W_{\mathcal{A}_G}^{\frac{1}{2},1;p}(0,T)$
with $D_{x}u(t,\omega )=Z_{t}(\omega )$. Moreover, we have
\begin{equation*}
u(T,\omega )=\xi (\omega ),\ u\in W_{\mathcal{A}_G}^{\frac{1}{2},1;p}(\eta )\ with\
\eta _{t}(\omega )=g(t,\omega ,u(t,\omega ),D_{x}u(t,\omega )).
\end{equation*}%
Namely $u$ is a $W_{\mathcal{A}_G}^{\frac{1}{2},1;p}$-solution to the path dependent PDE
(\ref{PPDE3'}-\ref{PPDE4'}).

(ii) Let $u\in W_{\mathcal{A}_G}^{\frac{1}{2},1;p}(0,T)$ be a $W_{\mathcal{A}_G}^{\frac{1}{2},1;p}$%
-solution to the path dependent PDE (\ref{PPDE3'}-\ref{PPDE4'}). Set
\begin{equation*}
K_{t}=u(t,\omega )+\int_{0}^{t}g(s,\omega ,u(s,\omega ),D_{x}u(s,\omega
))ds-\int_{0}^{t}D_{x}u(s,\omega )dB_{s}.
\end{equation*}%
Then $(u,D_{x}u,K)$ is a solution to the backward SDE (\ref{GBSDE-s4}).
\end{theorem}

Assume that the function $g(t,\omega,y,z):[0,T]\times \Omega_{T}\times
R\times R\rightarrow R$ satisfies the following assumption: there exists
some $\beta>1$ such that

\begin{description}
\item[($\mathcal{H}$1)] for any $y$,$z$, $g(t,\omega,y,z)\in M_{G}^{\beta
}(0,T)$;

\item[($\mathcal{H}$2)] $|g(t,\omega,y,z)-g(t,\omega,y^{\prime},z^{\prime
})|\leq L(|y-y^{\prime}|+|z-z^{\prime}|)$ for some constant $L>0$.
\end{description}

\begin{corollary}
\label{cor5.7} Assume $\xi \in L_{G}^{\beta }(\Omega _{T})$ and $g$
satisfies ($\mathcal{H}$1) and ($\mathcal{H}$2) for some $\beta >1$. Then,
for each $p\in (1,\beta )$, the path dependent PDE (\ref%
{PPDE3'}-\ref{PPDE4'}) has a unique $W_{\mathcal{A}_G}^{\frac{1}{2},1;p}$%
-solution $u\in W_{\mathcal{A}_G}^{\frac{1}{2},1;p}(0,T)$.
\end{corollary}

\begin{proof}
Uniqueness is straightforward from Theorem \ref{thm5.5} and Theorem \ref%
{thmA}.

Existence. By Theorem \ref{thmA} we know that the backward SDE (\ref%
{GBSDE-s4}) has a solution $(Y,Z,K)$. By the assumption ($\mathcal{H}$1) and
($\mathcal{H}$2), we conclude $g(t,\omega ,Y_{t}(\omega ),Z_{t}(\omega ))\in
M_{G}^{p}(0,T)$. So we get the existence from Theorem \ref{thm5.5}.
\end{proof}

\begin{corollary}
\label{cor5.6} $u\in S_{G}^{p}(0,T)$ is a $G$-martingale if and only if $%
u\in \mathcal{W}_{G}^{1,2;p}(0)$.
\end{corollary}

\begin{proof}
By the $G$-martingale decomposition theorem, $u\in S_{G}^{p}(0,T)$ is a $G$%
-martingale if and only if $u$ is a solution of backward SDE (\ref{GBSDE-s4}%
) with $f=0$.
\end{proof}

\subsection{Weak $G$-It\^o processes}

For a weak $G$-It\^{o} process
\begin{equation*}
u=u_{0}+\int_{0}^{t}\eta (s,\omega )ds+\int_{0}^{t}\zeta (s,\omega
)dB_{s}+K_{t},
\end{equation*}%
generally we don't know whether the above decomposition is unique. More
precisely, we can't distinguish $\int_{0}^{\cdot }\eta _{s}ds,\eta \in
M_{G}^{p}(0,T)$ from non-increasing $G$-martingales. In this section, we
confine $\eta $ in a subspace of $M_{G}^{p}(0,T)$ to guarantee that the
decomposition is unique.

For a step process $\eta $, set $\Vert \eta \Vert
_{\tilde{M}_{G}^{p}}^{2}=\int_{0}^{T}\mathbb{E}^{G}[|\eta _{s}|^{2}]ds$. Denote by $%
\tilde{M}_{G}^{p}(0,T)$ the completion of the collection of step processes with
respect to the norm $\Vert \cdot \Vert _{\tilde{M}_{G}^{p}}$.

\begin{proposition}\label{prop5.8}
For $p\geq1$, $\tilde{M}_{G}^{p}(0,T)$ is a subspace of $M_{G}^{p}(0,T)$.
\end{proposition}

\begin{proof}
Assume that $\{ \eta _{n}\} \subset M^{0}(0,T)$ is a Cauchy sequence w.r.t.
the norm $\Vert \cdot \Vert _{\tilde{M}_{G}^{p}}$ and that $\Vert \eta
_{n}\Vert _{M_{G}^{p}}$ converges to 0. We shall prove that $\Vert \eta
_{n}\Vert _{\tilde{M}_{G}^{p}}$ converges to 0. Actually, since $\{ \eta
_{n}\} \subset M^{0}(0,T)$ is a Cauchy sequence w.r.t. the norm $\Vert \cdot
\Vert _{\tilde{M}_{G}^{p}}$, there exists a process $\eta $ such that $\Vert
\eta _{n}-\eta \Vert _{\tilde{M}_{G}^{p}}\rightarrow 0$. Then we have $\Vert
\eta \Vert _{M_{G}^{p}}=0$ since $\Vert \eta _{n}\Vert _{M_{G}^{p}}$
converges to 0. For $m\in \mathbb{N}$, set $h=T/m$ and
\begin{equation*}
\eta _{t}^{h}=\sum_{k=1}^{m-1}1_{(kh,(k+1)h]}(t)\frac{1}{h}%
\int_{(k-1)h}^{kh}\eta _{s}ds.
\end{equation*}

Clearly, we have $\Vert \eta ^{h}\Vert _{\tilde{M}_{G}^{p}}=0$.
Consequently, we have
\begin{eqnarray*}
\Vert \eta \Vert _{\tilde{M}_{G}^{p}} &\leq &\Vert \eta ^{h}-\eta \Vert _{%
\tilde{M}_{G}^{p}} \\
&\leq &\Vert \eta ^{h}-\eta _{n}^{h}\Vert _{\tilde{M}_{G}^{p}}+\Vert \eta
_{n}-\eta _{n}^{h}\Vert _{\tilde{M}_{G}^{p}}+\Vert \eta _{n}-\eta \Vert _{%
\tilde{M}_{G}^{p}} \\
&\leq &2\Vert \eta _{n}-\eta \Vert _{\tilde{M}_{G}^{p}}+\Vert \eta _{n}-\eta
_{n}^{h}\Vert _{\tilde{M}_{G}^{p}}.
\end{eqnarray*}%
First letting $h$ converge to 0, then letting $n$ go to infinity, we have $%
\Vert \eta \Vert _{\tilde{M}_{G}^{p}}=0$. So
\begin{equation*}
\Vert \eta _{n}\Vert _{\tilde{M}_{G}^{p}}\leq \Vert \eta _{n}-\eta \Vert _{%
\tilde{M}_{G}^{p}}\rightarrow 0.
\end{equation*}
\end{proof}

\begin{lemma}
\label{lem-uni} Assume $\int_{0}^{t}\eta _{s}ds+K_{t}=L_{t}$, where $\eta
\in \tilde{M}_{G}^{p}(0,T)$, $K_{t},L_{t}$ are non-increasing $G$%
-martingales with $K_{T},L_{T}\in L_{G}^{p}(\Omega _{T})$ for some $p>1$. Then we have $%
\int_{0}^{t}\eta _{s}ds=0$ and $K_{t}=L_{t}$.
\end{lemma}

\begin{proof}
Let $\zeta \in \tilde{M}_{G}^{p}(0,T)$. We claim that $A=0$ if $%
A_{t}:=\int_{0}^{t}\zeta _{s}ds$ is a $G$-martingale. In fact $A_{t}$ must
be a non-increasing $G$-martingale by $G$-martingale decomposition theorem.
For $n\in \mathbb{N}$, set $h=T/n$ and
\begin{align*}
\hat{\zeta}_{t}^{n}& =\sum_{k=0}^{n-1}1_{(kh,(k+1)h]}(t)\frac{1}{h}%
\int_{kh}^{(k+1)h}\zeta _{s}ds, \\
\check{\zeta}_{t}^{n}& =\sum_{k=1}^{n-1}1_{(kh,(k+1)h]}(t)\frac{1}{h}%
\int_{(k-1)h}^{kh}\zeta _{s}ds.
\end{align*}%
For $t\in (kh,(k+1)h]$, we have $\mathbb{E}^{G}[\hat{\zeta}_{t}^{n}-\check{%
\zeta}_{t}^{n}]=\frac{1}{h}\mathbb{E}^{G}[-(A_{kh}-A_{(k-1)h})]$. So

\begin{align*}
0\leftarrow \{T^{p-1}\int_{0}^{T}\mathbb{E}^{G}[|\hat{\zeta}_{t}^{n}-\check{\zeta }%
_{t}^{n}|^{p}]dt\}^{1/p} & \geq \int_{0}^{T}\mathbb{E}^{G}[\hat{\zeta}%
_{t}^{n}-\check{\zeta}_{t}^{n}]dt \\
& =\sum_{k=1}^{n-1}\mathbb{E}^{G}[-(A_{kh}-A_{(k-1)h})] \\
& \geq \mathbb{E}^{G}[-A_{\frac{(n-1)T}{n}}]\rightarrow \mathbb{E}%
^{G}[-A_{T}].
\end{align*}
Assume $\int_{0}^{t}\eta_{s}ds+K_{t}=L_{t}$. Since $L_{t}$ is
non-increasing, $\tilde{L}_{t}:=\int_{0}^{t}\eta_{s}^{+}ds+K_{t}$ is also
non-increasing.  By this we have $0\geq%
\mathbb{E}_s^{G}[-\int_{s}^{t}\eta_{s}^{+}ds]\geq \mathbb{E}_s%
^{G}[K_{t}-K_{s}]=0 $. So $-\int_{0}^{t}\eta_{s}^{+}ds$ is a $G$-martingale,
which implies that $\int_{0}^{t}\eta_{s}^{+}ds=0$. By the same arguments, we
have $\int_{0}^{t}\eta_{s}^{-}ds=0$. By Proposition \ref {prop5.8}, we have $\|\eta\|_{\tilde{M}^p_G}=0$.
\end{proof}

\section{Appendix: Backward SDEs driven by $G$-BM}

In [HJPS2012] the authors studied the backward stochastic differential
equations driven by a $G$-Brownian motion $(B_{t})_{t\geq0}$ in the
following form:
\begin{equation}
Y_{t}=\xi+\int_{t}^{T}f(s,Y_{s},Z_{s})ds-%
\int_{t}^{T}Z_{s}dB_{s}-(K_{T}-K_{t}).  \label{equA}
\end{equation}
where $K$ is a non-increasing $G$-martingale.

The main result in [HJPS2012] is the existence and uniqueness of a solution $%
(Y,Z,K)$ for equation (\ref{equA}) in the $G$-framework under the following
assumption: there exists some $\beta>1$ such that

\begin{description}
\item[(H1)] for any $y,z$, $f(\cdot,\cdot,y,z)\in M_{G}^{\beta}(0,T)$;

\item[(H2)] $|f(t,\omega,y,z)-f(t,\omega,y^{\prime},z^{\prime})|\leq
L(|y-y^{\prime}|+|z-z^{\prime}|)$ for some $L>0$.
\end{description}

\begin{definition}
\label{defA1} Let $\xi \in L_{G}^{\beta}(\Omega_{T})$ and $f$ satisfy (H1)
and (H2) for some $\beta>1$. A triplet of processes $(Y,Z,K)$ is called a
solution of equation (\ref{equA}) if for some $1<\alpha \leq \beta$ the
following properties hold:

\begin{description}
\item[(a)] $Y\in S_{G}^{\alpha}(0,T)$, $Z\in H_{G}^{\alpha}(0,T)$, $K$ is a
non-increasing $G$-martingale with $K_{0}=0$ and $K_{T}\in L_{G}^{\alpha
}(\Omega_{T})$;

\item[(b)] $Y_{t}=\xi+\int_{t}^{T}f(s,Y_{s},Z_{s})ds-%
\int_{t}^{T}Z_{s}dB_{s}-(K_{T}-K_{t})$.
\end{description}
\end{definition}

The main result in [HJPS2012] is the following theorem:

\begin{theorem}
\label{thmA} Assume that $\xi \in L_{G}^{\beta}(\Omega_{T})$ and $f$
satisfies (H1) and (H2) for some $\beta>1$. Then equation (\ref{equA}) has a
unique solution $(Y,Z,K)$. Moreover, for any $1<\alpha<\beta$ we have $Y\in
S_{G}^{\alpha}(0,T)$, $Z\in H_{G}^{\alpha}(0,T)$ and $K_{T}\in L_{G}^{\alpha
}(\Omega_{T})$.
\end{theorem}


\renewcommand{\refname}{\large References}{\normalsize \ }

\end{document}